\documentclass[11pt]{amsart}

\usepackage{amsmath}
\usepackage{amssymb}
\usepackage{epsfig}
\usepackage{amsfonts}
\usepackage{ytableau}
\ytableausetup{boxsize=1.2em}

\newcommand{\excise}[1]{}%$\star$\textsc{#1}$\star$}

\newtheorem{thm}{Theorem}[section]
\newtheorem{lemma}[thm]{Lemma}

\newtheorem{cor}[thm]{Corollary}
\newtheorem{prop}[thm]{Proposition}

\newtheorem{conv}[thm]{Convention}

\newtheorem{ex}[thm]{Example}

\newtheorem{Warn}[thm]{Caution}

    {\end{Example}}
    {\end{Remark}}

                     {\end{minipage}\end{center}\smallskip}

\newenvironment{Abs}{\smallskip\begin{center}\begin{minipage}{13.8cm}}%
                     {\end{minipage}\end{center}\smallskip}

\numberwithin{equation}{section}

\def\wh{\widehat}

\def\emp{\nothing}
\def\sq{\square}

\def\ov{\overline}

\def\la{\lambda}
\def\ga{\gamma}
\def\si{\sigma}
\def\de{\delta}

\def\al{\alpha}
\def\be{\beta}
\def\om{\omega}

\def\cd{\mathcal D}

\def\cP{\mathcal P}

\def\CT{\mathcal T}

\def\ssu{\subset}

\def\<{\langle}
\def\>{\rangle}

\def\rc{ {\text {\rm c}  } }
\def\rh{ {\text {\rm h}  } }
\def\rd{ {\text {\rm d}  } }

\def\GL{ {\text {\rm GL} } }

\def\vt{\vartheta}

\def\0{{\mathbf 0}}

\def\nothing{\varnothing}

\def\.{\hskip.06cm}
\def\ts{\hskip.03cm}

\def\ze{{\zeta}}

\def\sign{{\rm sign}}
\def\sgn{{\rm sgn}}

\def\SP{{\textsc{\#P}}}

\def\pq{\delta}

\def\nin{\noindent}
\def\CC{\mathrm{CA}}
\def\CT{\mathrm{CT}}

\def\cD{\cd}

\def\ovl{\al}
\def\ovm{\be}
\def\ovn{\ga}
\def\ovr{\eta}
\def\wht{\wh}

\def\parti{\text{\small \textmd{P}}}

\def\dwn{{\hskip-.03cm \downarrow}\ts}
\def\upa{{\hskip-.03cm \uparrow}\ts}

\def\da{\vt}

\begin{document}
\title{Bounds on the Kronecker coefficients}

\author[Igor~Pak]{ \ Igor~Pak$^\star$}

\author[Greta~Panova]{ \ Greta~Panova$^\star$}

\date{\today}

\thanks{\thinspace ${\hspace{-.45ex}}^\star$Department of Mathematics,
UCLA, Los Angeles, CA 90095, USA; \ts
\texttt{\{pak,panova\}@math.ucla.edu}}

\maketitle

\begin{Abs}{\footnotesize {\sc Abstract.} \ts
We present several upper and lower bounds on the Kronecker coefficients
of the symmetric group.   We prove $k$-stability of the Kronecker
coefficients generalizing the (usual) stability, and giving a new upper bound.
We prove a lower bound via the characters of $S_n$.  We apply these and
other results to generalize Sylvester's unimodality of $q$-binomial coefficients $\binom{n}{k}_q$ as polynomials in $q$: we derive explicit sharp bounds on the differences of their consecutive coefficients.
}
\end{Abs}

\bigskip

\section{Introduction}

\nin
The \emph{Kronecker coefficients} are perhaps the most challenging, deep
and mysterious objects in Algebraic Combinatorics.  Universally admired,
they are beautiful, unapproachable and barely understood.
For decades since they were introduced by Murnaghan in~1938, the field
lacked tools to study them, so they remained largely out of reach.
However, in recent years a flurry of activity led to significant advances,
spurred in part by the increased interest and applications to other fields.
We refer to~\cite{PP-future} for a detailed survey of these advances and
further references.

In this paper, we initiate the study of asymptotics of the Kronecker coefficients.
We are motivated by applications to the \ts \emph{$q$-binomial $($Gaussian$\ts)$
coefficients}, and by connections to the \emph{Geometric Complexity
Theory} (see~$\S$\ref{ss:fin-GCT}).  The tools are
based on technical advances in combinatorial representation theory
obtained in recent years, see~\cite{BOR2,CDW,CHM,Man,Val2}, and our own
series of papers~\cite{PP_s,PP,PP_c,PPV}.  In fact, here we give several
far reaching extensions of our earlier work.

\smallskip

The \emph{Kronecker coefficients} \ts $g(\la,\mu,\nu)$ \ts are defined by:
\begin{equation}\label{eq:kron-def}
\chi^\la \ts \otimes \ts \chi^\mu \, = \, \sum_{\nu \vdash n}
\, g(\la,\mu,\nu)\. \chi^\nu\., \quad \text{where} \ \ \la,\mu \vdash n\ts,
\end{equation}
where $\chi^\al$ denotes the character of the irreducible representation
%$\mathbb{S}^\al$
of $S_n$ indexed by partition~$\al\vdash n$.  They are integer and nonnegative
by definition, have full~$S_3$ symmetry, and satisfy a number of further
properties (see~$\S$\ref{ss:basic-kron}). In contrast with their ``cousins''
\emph{Littlewood--Richardson $($LR$)$ coefficients},
%\ts $c^\la_{\mu\nu}$,
they lack a combinatorial interpretation or any meaningful positive formula,
and thus harder to compute and to estimate (cf.~$\S$\ref{ss:fin-LR}).

\smallskip

Our first result is a new type of stability of Kronecker coefficients:

\begin{thm}[$k$-stability] \label{t:k-stab-limit}
Let $\la,\mu,\nu \vdash n$ and $k\ge 1$ be fixed. Then the function \ts
$G_k(t)=G(\la,\mu,\nu; \ts t)$ \ts defined by
$$
G_k(t) \, := \, g\bigl(\la + (t^k), \ts \mu +(t^k), \ts \nu + (tk)\bigr)\.,
$$
is bounded and monotone increasing, as~$t$ grows.
\end{thm}

Denote by $\ov{g}_k(\la,\mu,\nu)$ the limit of~$G_k(t)$, for $t$ large enough.
We call it the \emph{$k$-stable Kronecker coefficient}.  For $k=1$,
the integer $\ov{g}_1(\la,\mu,\nu)$ is called the (usual) stable (or reduced)
Kronecker coefficient.  It has been a subject of intense study
going back to Murnaghan and Littlewood, while its monotonicity
property is due to Brion (see $\S$\ref{ss:fin-stab}).  For $\la=\mu$,
the existence of the limit as in the theorem was recently observed by Vallejo, using
a novel diagrammatic approach specific to tensor squares~\cite{Val2}.\footnote{After 
this paper was finished, Vallejo informed us that he was also able to use his 
diagrammatic approach to derive $k$-stability in full generality, % (personal communication), 
to be included in a revised version of~\cite{Val2}. Also, John Stembridge reported to us 
that he obtained a further generalization (personal communication). }
In a different direction, for a special case of rectangular
diagrams $\la=\mu=(k^\ell)$, this limit was computed by Manivel~\cite{Man}.
Let us mention that the proof of Theorem~\ref{t:k-stab-limit} uses the technical
\emph{Reduction Lemma}~\ref{l:reduction} proved in~\cite{PP_c} for
computational complexity applications.

% We apply the \emph{$k$-stable Kronecker coefficients} \ts
% $\ov{g}_k(\la,\mu,\nu)$ \ts to obtain upper bounds...

\smallskip

Our second result is a lower bound of the Kronecker coefficients $g(\la,\mu,\mu)$
for multiplicities in tensor squares of self-conjugate partitions:

\begin{thm}\label{t:char_effective}
Let $\mu=\mu'$ be a self-conjugate partition and let $\wh\mu=(2\mu_1-1,2\mu_2-3,\ldots)\vdash n$
be the partition of its principal hooks. Then:
$$
g(\la, \mu,\mu) \, \geq \, \bigl|\ts\chi^\la[\wh\mu] \ts \bigr|\,, \quad \text{for every}  \quad \la \vdash n\ts.
$$
\end{thm}

While it is relatively easy to obtain various upper bounds on the Kronecker
coefficients (see Section~\ref{s:gen}), this is the only general lower bound that we know.
The theorem strengthens a qualitative result \ts $g(\la, \mu,\mu) \ge 1$ \ts given
in~\cite[Lemma~1.3]{PPV}, used there to prove a special case of the Saxl conjecture
(see $\S$\ref{ss:fin-char}).  We use the bound to give a new proof of
Stanley's Theorem~\ref{t:stanley-unimod}, from~\cite{Sta-unim}.%; the only other proof we know is the original proof based on the Weak Lefschetz Property~\cite{Sta-unim}.
% (see~$\S$\ref{ss:fin-unimod}).
%
% Other applications of...

\smallskip

Our final result is motivated by an application of bounds for Kronecker
coefficients to the \emph{$q$-binomial coefficients}, defined as:
$$
\binom{m+\ell}{m}_q
\, = \ \. \frac{(q^{m+1}-1)\. \cdots\. (q^{m+\ell}-1)}{(q-1)\.\cdots\. (q^{\ell}-1)}
\ \. = \, \, \sum_{n=0}^{\ell\ts m} \, \. p_n(\ell,m) \. q^n\ts.
$$
In~1878, \emph{Sylvester} proved \emph{unimodality} of the coefficients:
$$
% (\star )\quad
p_0(\ell,m)\. \le \. p_1(\ell,m)\. \le \. \ldots \. \le \. p_{\lfloor\ell\ts m/2\rfloor}(\ell,m) \. \ge \. \ldots \. \ge \. p_{\ell\ts m}(\ell,m)\ts,
$$
see~\cite{Syl}.  In~\cite{PP_s}, we used the Kronecker coefficients to prove \emph{strict unimodality}:
\begin{equation}\label{eq:strict-unim}
p_k(\ell,m) \ts - \ts p_{k-1}(\ell,m) \. \ge 1\., \
\quad \text{for} \ \quad 2\le k \le \ell\ts m/2\ts, \ \. \ell, \ts m \ge 8\ts.
\end{equation}
Here we further strengthen this result as follows.

\begin{thm}\label{t:qbin-main}
There is a universal constant $A>0$, such that for all $m\ge \ell\ge 8$ and
$2\le k\le \ell \ts m/2$, we have:
$$
p_k(\ell,m) \ts - \ts p_{k-1}(\ell,m) \, \ts > \, A \, \frac{2^{\sqrt{s}}}{s^{9/4}}\,, \quad \ \text{where} \quad \.
s=\min\{2k,\ell^2\}\ts.
$$
\end{thm}

We compare this lower bound with upper bounds in~$\S$\ref{ss:qbin-upper}
(see also~$\S$\ref{ss:fin-qbin}). The proof of the theorem has several ingredients.
We use the above mentioned Stanley's theorem, an extension of analytic
estimates in the proof of \emph{Almkvist's Theorem}
(Theorem~\ref{t:almkvist-unimod}), and \emph{Manivel's inequality}
for the Kronecker coefficients (Theorem~\ref{t:manivel}).
Most crucially, we use the following connection between the
$q$-binomial and the Kronecker coefficients:

\begin{lemma}[Two Coefficients Lemma]
\label{l:g_partitions}
Let $n=\ell \ts m$, \ts $\tau_k=(n-k,k)$, where $1\leq k\leq n/2$.
% Denote \ts $\pq_k(\ell,m) \ts = \ts p_k(\ell,m) -  p_{k-1}(\ell,m)$.
Then:
$$g\bigl(m^\ell,m^\ell,\tau_k\bigr) \, = \, p_k(\ell,m) \. - \. p_{k-1}(\ell,m)\ts.
$$
% $$g(m^\ell,m^\ell,\tau_k) \, = \, \pq_{k}(\ell,m)\ts. $$
\end{lemma}

This simple by very useful lemma was first proved in~\cite[$\S 4$]{Val2} 
and later in~\cite{PP_s}, but is implicit in~\cite{MY,PP}.  Note that it 
immediately implies Sylvester's unimodality theorem.

\smallskip

The rest of the paper is structured as follows.  We begin with a quick recap
of definitions, notations and some basic results we are using (Section~\ref{s:basic}).
In Section~\ref{s:gen} we give general upper bounds on Kronecker coefficients.
We then prove the above three theorems, in this order (sections~\ref{s:stab}--\ref{s:qbin}).
We conclude with final remarks and open problems (Section~\ref{s:fin}).

\smallskip

Finally, let us mention that the paper is very far from being self-contained.
However, we have made an effort to quote explicitly all results and techniques we are using,
%Although we use a number of results and techniques developed in the field, we made an effort to quote all of them explicitly,
so the paper should be accessible
to a novice reader.

\bigskip

\section{Definitions and basic results}\label{s:basic}

\subsection{Partitions and Young diagrams}\label{ss:basic-part}
We adopt the standard notation in combinatorics of partitions
and representation theory of~$S_n$, as well as the theory of symmetric functions
(see e.g.~\cite{Mac,Sta}).

Let $\cP$ denote the set of integer partitions $\la = (\la_1,\la_2,\ldots)$.  We write
$|\la|=n$ and $\la\vdash n$, for $\la_1+\la_2+\ldots = n$. Let $\cP_n$ the set of
all $\la\vdash n$, and let $\parti(n)=|\cP_n|$ the number of partitions of~$n$.
We use $\ell(\la)$ to denote the number of parts of~$\la$, and $\la'$ to denote
the conjugate partition.  Let $[\la]$ denotes the \emph{Young diagram} of partition~$\la$.
We write $\la \ssu \mu$ if $[\la] \ssu [\mu]$.  Similarly, we use $\la \cap \mu$
and $\la \cup \mu$ to indicate corresponding operations on Young diagrams.

Define addition of partitions $\al,\be \in \cP$ to be their addition as vectors:
$$\alpha+\beta \. = \. (\alpha_1+\beta_1, \ts \alpha_2+\beta_2,\ts \ldots)\ts.$$
Let $u=(i,j)$ be a square in~$[\al]$, $1\le i\le \ell(\al)$, $1\le j \le \la_i$.
Denote by $\rc(u) = j-i$ the \emph{content} and by
$\rh(u) = \lambda_i-j +\lambda'_j-i+1$ the \emph{hook length} of~$u$.

\subsection{Irreducible characters} \label{ss:basic-char}
We use $\chi^\la$ to denote the character of the irreducible representation
of $S_n$ corresponding to~$\la$.
Recall the \emph{hook length formula}
for the dimension:
\begin{equation}\label{eq:hlf}
f^\la \, = \, \chi^\la(1^n) \, = \, \frac{n!}{\prod_{u\in [\la]}\.\rh(u)}\..
\end{equation}

Similarly, for the dimension $d_\la(m)$ of the irreducible representation of $\GL_m(\mathbb{C})$
corresponding to~$\la$, $\ell(\la) \le m$, we have the \emph{hook content formula}
(see e.g.~\cite[\S 7.21.4]{Sta})
\begin{equation}\label{eq:hcf}
d_\la(m) \, = \, s_\la(1^m) \, = \, \prod_{u \in \la} \. \frac{ m+c(u)}{h(u)}\..
\end{equation}

\subsection{Kronecker coefficients}\label{ss:basic-kron}
It is well known that
%\begin{equation}\label{eq:char-kron}
$$
g(\la,\mu,\nu) \, = \, \frac{1}{n!} \. \sum_{\omega \in S_n} \. \chi^\la(\omega)\ts \chi^\mu(\omega)\ts \chi^\nu(\omega)\ts.
$$
%\end{equation}
This implies that Kronecker coefficients have full $S_3$ group of symmetry:
%\begin{equation}\label{eq:kron-S3}
$$
g(\la,\mu,\nu) \, = \, g(\mu,\la,\nu) \, = \, g(\la,\nu,\mu) \, = \, \ldots % \,,
$$
%\end{equation}
In addition, recall that $\chi^\la \otimes \sign = \chi^{\la'}$, where the \ts
``$\sign$'' \ts denotes character corresponding to partition~$(1^n)$.  This implies
\begin{equation}\label{eq:kron-conj}
g(\la,\mu,\nu) \, = \, g(\la',\mu',\nu) \, = \, g(\la',\mu,\nu') \, = \, g(\la,\mu',\nu')\ts.
\end{equation}
Another useful formula is given by the \emph{generalized Cauchy identity} (see e.g. \cite[Ex. 7.78(f)]{Sta})
\begin{equation}\label{cauchy}
\sum_{\la,\ts\mu,\ts\nu} \, g(\la,\mu,\nu)\. s_{\la}(x)\ts s_{\mu}(y)\ts s_{\nu}(z)
\, = \, \prod_{i,j,k} \. \frac{1}{1-x_i\ts y_j\ts z_k}\.,
\end{equation}
where the summations is over all triples of partitions of the same size $|\la|=|\mu|=|\nu|$.
This formula was used in~\cite{PP_c} to obtain the following alternating sign formula
(see also~\cite{CDW} where a similar formula was found).

Let $\al,\be,\ga \vdash n$ be partitions with lengths $\ell(\la)=a$, $\ell(\mu)=b$, and $\ell(\nu)=c$.
Denote by $\CC(\alpha,\beta,\gamma)$
the number of $3$-dimensional \emph{contingency arrays} of size $[a\times b \times c]$,
with $2$-dimensional marginal sums $\alpha,\beta,\gamma$.  Similarly, denote 
by $\CC^\ast(\alpha,\beta,\gamma)$ the number of such arrays with \ts $0$--$1$~entries. 

\begin{thm}[\cite{PP_c}]  \label{t:cont}
Let $\la,\mu,\nu \vdash n$ be partitions with lengths $\ell(\la)=a$, $\ell(\mu)=b$, and $\ell(\nu)=c$.
Denote by $\delta_k=(k-1,\ldots,1,0)$, a partition of $\binom{k}{2}$.  Then:
$$g(\la,\mu,\nu) \, =\. \sum_{\si \in S_a,\.\om \in S_b,\.\pi \in S_c} \sgn(\si,\om,\pi) \cdot \CC\bigl(\la+\delta_a-\si\cdot\de_a,\mu+\delta_b-\om\cdot\de_b,\nu+\delta-\pi\cdot\de_c\bigr),$$
where $\sgn(\si,\om,\pi) \in \{\pm 1\}$ is the product of signs of three permutations.
\end{thm}

Here $\si \cdot \al$ denotes the permutation $(\alpha_{\si(1)},\ldots,\al_{\si(k)})$ of the parts of  the partition $\al=(\al_1,\ldots,\al_k)$ according
to $\si \in S_k$.  Note that when $a,b,c=O(1)$, the theorem implies that the Kronecker coefficients
can be computed in polynomial time via Barvinok's algorithm~\cite{CDW,PP_c}.

\smallskip

The following \emph{Manivel's inequality} is given in~\cite{Man}. It can be viewed as an effective
version of the \emph{semigroup property} (see~\cite{CHM,Zel}).

\begin{thm}[\cite{Man}]  \label{t:manivel}
Suppose $\la,\mu,\nu,\al,\be,\ga$ are partitions of~$n$, such that the Kronecker coefficients
$\ts g(\la,\mu,\nu), \ts g(\al,\be,\ga)> 0$. Then:
$$g(\la+\al,\mu+\be,\nu+\ga)\. \geq \. \max\ts\bigl\{ g(\la,\mu,\nu), \ts g(\al,\be,\ga)\bigr\}\ts.$$
\end{thm}

\smallskip

\subsection{Partition asymptotics}\label{ss:basic-asympt}
Denote by \ts $\parti'(n) = \parti(n) - \parti(n-1)$ the number of partitions into
parts~$\ge 2$. Recall that $\parti'(n)\ge 1$ for all $n\ge 2$, and the following
\emph{Hardy--Ramanujan} and \emph{Roth--Szekeres} formulas,
respectively:
\begin{equation} \label{HR asymptotics}
 \parti(n) \. \sim \. \frac{1}{4\sqrt{3}\ts n} \,\, e^{\pi \sqrt{\frac{2}{3}\ts n}} \,,\quad
  \parti'(n) \. \sim \. \frac{\pi}{\sqrt{6\ts n}} \,\parti(n)  \, \ \quad \text{as } \ n\to\infty\.,
 \end{equation}
see~\cite{RS} (see also~\cite[p.~59]{ER}).

Denote by $b_k(n)$ the number of partitions of $k$ into distinct odd parts~$\le 2n-1$.  We have:
$$
\prod_{i=1}^{n}\. \ts \bigl(1+q^{2i-1}\bigr)\ \. = \, \, \sum_{k=0}^{n^2} \, \. b_k(n) \. q^k\ts.
$$
\begin{thm}[Almkvist] \label{t:almkvist-unimod}
The following sequence is symmetric and unimodal:
$$
(\lozenge)\quad
b_{2}(n)\ts, \, b_{3}(n)\ts, \, \ldots \,, \, b_{n^2-2}(n)\ts.
$$
\end{thm}

%Note that for $n\ge 14$, we have $b_{25}(n)=??$ and $b_{26}(n)=??$, so unimodality fails.

% \vskip.6cm

\bigskip

\section{General upper bounds}\label{s:gen}

In this section we present several natural upper bounds on
Kronecker coefficients some of which we use later in the paper.
While the proofs are not difficult, we were unable to locate
them anywhere in the literature.

\subsection{Bounds via dimension} \label{ss:gen-dim}
By definition of Kronecker coefficients~\eqref{eq:kron-def},
we have the following trivial upper bound.\footnote{Note that a superficially
similar bound also holds for LR coefficients ($\S$\ref{ss:gen-LR}).}

\begin{prop} \label{p:upper-dim}
For every $\la, \mu, \nu \vdash n$, we have:
$$(\triangledown) \qquad
g(\la,\mu,\nu) \, \leq \, \frac{f^{\la}\. f^{\mu}}{f^{\nu}}\,.
$$
\end{prop}

Now the r.h.s.~can be computed via the hook length formula~\eqref{eq:hlf}.
By the symmetry of Kronecker coefficients, we have a surprisingly simple formula,
see e.g.~\cite[Exc.~7.83]{Sta}. 

\begin{cor} \label{c:upper-min}
For every $\la, \mu, \nu \vdash n$, we have \.
$g(\la,\mu,\nu) \ts \leq \ts \min\bigl\{f^\la,\ts f^{\mu},\ts f^{\nu}\bigr\}$.
\end{cor}

\smallskip

In particular, for $n=\ell\ts m$, $\la=(m^\ell)$ and $\mu=\tau_k=(\ell\ts m-k,k)$,
$k\le n/2$, we have:
$$
g(m^\ell, m^\ell, \tau_k) \, \le \. f^{\tau_k} \. = \, \binom{n}{k} \ts - \ts \binom{n}{k-1} \, \le \, \frac{n^k}{k!}\..
$$
More specifically, for $n=2\ts k$, we have
\begin{equation}\label{eq:upper-cat}
g(m^\ell, m^\ell, \tau_k) \, \le \,  \frac{1}{k+1}\ts \binom{2\ts k}{k} \, \sim \, \frac{1}{\sqrt{\pi} \. k^{3/2}} \, 4^k\ts.
\end{equation}

\subsection{Bounds via Schur functions}   \label{ss:gen-schur}
We can also bound the Kronecker coefficients via the Schur-Weyl duality and the hook-content formula~\eqref{eq:hcf} for dimension of the irreducible $\GL_\ell$-modules (cf.~$\S$\ref{ss:gen-dim} above).  Here we use the Schur function language.

\begin{prop}\label{p:upper-GL}
Let $\la,\mu,\nu \vdash n$, $\ell(\la)=a$, $\ell(\mu)=b$  and \ts $\ell(\nu)=c$.  Then:
$$(\star) \qquad
g(\la,\mu,\nu) \, \leq \, \frac{d_\nu(ab)}{d_\la(a)\.d_\mu(b)} \..$$
\end{prop}

\begin{proof}
First, the generalized Cauchy identity~\eqref{cauchy} can be reinterpreted using the (usual) Cauchy identity with the variables $z$ and $u=xy=(x_1y_1,x_1y_2,\ldots,x_2y_1,\ldots)$, as follows:
$$
\prod_{i,j,k} \. \frac{1}{1-x_iy_jz_k} \, = \prod_{i,r} \. \frac{1}{1-z_iu_r} \,=
\, \sum_{\al\in \cP} \. s_{\al}(z)\ts s_{\al}( xy)\..
$$
%where $yz=(\ldots,y_iz_j,\ldots)$.
Comparing the left-hand side of~\eqref{cauchy}
 and the right-hand side above, and equating coefficients at $s_{\nu}(z)$, we see that
$$\sum_{\mu,\ts\nu} \, g(\la,\mu,\nu)\. s_\la(x)\ts s_\mu(y)\,\. = \, s_{\nu}(xy)\..
$$
Hence, for any nonnegative values of $x,y$, we have:
\begin{equation}\label{eq:kron-schur}
g(\la,\mu,\nu) \, \leq \, \frac{s_{\nu}(xy)}{s_{\la}(x)\ts s_{\mu}(y)}\..
\end{equation}
In particular, setting $x=(1^{a})$ and $y= (1^{b})$, we have $xy=(1^{ab})$.
Combining~\eqref{eq:kron-schur} and the first equality in~\eqref{eq:hcf},
we get the result.
\end{proof}

Now, the r.h.s. of~$(\star)$ can be computed via the hook-content formula~\eqref{eq:hcf}.
For example, for $n=\ell\ts m$, $\la=(m^\ell)$ we have $d_\la(\ell)=1$, and
$$
d_\la(N) \, = \, \frac{(N+m-1)! \cdots (N+m-\ell)! \.
(\ell-1)! \cdots 1!}{(N-1)!\cdots (N-\ell)!\. (m+\ell-1)! \cdots m!}\..
$$
For $\ell=2$, $\la = (m,m)$, the upper bound~$(\star)$ gives
\begin{equation}\label{eq:upper-schur}
g(\la,\ts \la, \ts \la) \, \le \, d_{m^2}(4) \, = \,
\frac{(m+3)! \ts (m+2)!}{3! \ts 2! \ts (m+1)! \ts m!} \, \sim \, \frac{m^4}{12}
\ts,
\end{equation}
which is much smaller than the exponential bound~\eqref{eq:upper-cat}.
On the other hand, a similar explicit calculation for $\ell = m$ shows
that the bound~$(\triangledown)$ is much better than~$(\star)$ in that case.
More generally, we have the following compact formula for an explicit upper bound.

\begin{cor}\label{c:gen-upper-binom}
Let
$\la,\mu,\nu \vdash n$, $\ell(\la)=a$, $\ell(\mu)=b$  and \ts $\ell(\nu)=c$.   Then:
$$
g(\la,\mu,\nu) \; \leq \; \prod_{i=1}^c \binom{\nu_i-i+ab}{\nu_i}\ts .
$$
\end{cor}
\begin{proof}
We have that $d_{\la}(a)\geq 1$, $d_{\mu}(b)\geq 1$. Applying the hook-content formula by collecting the products over boxes in row $i$, we have
$$d_{\nu}(ab) =\prod_{i=1}^c  \;  \frac{(\nu_i-i+ab)!}{(ab-i)!} \cdot \frac{1}{ \prod_{j=1}^{\nu_i}  (\nu_i-j+1 +\nu'_j-i)}.$$
In particular, the products of the hook-length in row $i$ are greater than $\nu_i!$, so
$$d_{\nu}(ab) \leq \prod_{i=1}^c \frac{(\nu_i-i+ab)!}{(ab-i)!} \frac{1}{\nu_i!}  =\prod_{i=1}^c \binom{\nu_i-i+ab}{\nu_i}.$$
Applying Proposition~\ref{p:upper-GL} gives the bound.
\end{proof}

\subsection{Bounds via contingency arrays}  \label{ss:gen-cont}
We recall the following recent result by Avella-Alaminos and Vallejo~\cite[$\S$2]{AV}, 
based on the older work of Snapper~\cite{Sna}. 

\begin{thm}[\cite{AV}]\label{t:upper}
Let $\la,\mu,\nu \vdash n$.  Then:
% $\ell(\la)=a$, $\ell(\mu)=b$ and $\ell(\nu)=c$.
$$
g(\la,\mu,\nu) \, \leq \, \CC(\la,\mu,\nu)\. \quad \text{and} \quad 
g(\la,\mu,\nu) \, \leq \, \CC^\ast(\la',\mu,\nu)\..
$$
\end{thm}

It is useful to compare the theorem with  Theorem~\ref{t:cont}, where $\CC(\la,\mu,\nu)$
is the leading and also the maximal term.  The last fact is proved in~\cite{PP-future}.  

\smallskip

Note also that $\CC(\cdot)$ and $\CC(\cdot)$ are difficult to estimate except for 
a few special cases (see~$\S$\ref{ss:fin-cont})
For example, for $\la=(m,m)$ as above, we have $\ell=2$ and a direct calculations gives
$$
g(\la,\la,\la) \, \le \, \CC(\la,\la,\la)\, \sim \,  \frac{m^4}{12} \.,
$$
which is asymptotically equal to the bound~\eqref{eq:upper-schur} above 
(cf.~\ref{ss:fin-two-rows}).  On the other hand, 
for the same $\la=(m,m)$ and $n=2\ts m = k^2$, the second inequality in the theorem
and the symmetry~\eqref{eq:kron-conj} gives 
$$
g\bigl(\la,k^k,k^k\bigr) \, \le \, \CC^\ast\bigl(\la,k^k,k^k\bigr)\ts.
$$
Since 
$$
\binom{n}{n/2} \, \le \, \CC^\ast\bigl(\la,k^k,k^k\bigr) \, \le \, \binom{n}{n/2}^2\.,
$$
the theorem gives a weaker bound than the dimension bound~\eqref{eq:upper-cat}.  
At the same time, Corollary~\ref{c:gen-upper-binom} gives an even weaker bound
$$
g\bigl(\la,k^k,k^k\bigr) \, \le \, \binom{3n/2}{n/2}^2\ts.
$$

\bigskip

\section{Bounds via stability}\label{s:stab}

\subsection{$k$-stability of Kronecker coefficients}\label{ss:stab-k}
We begin with an effective version of $k$-stability.

\begin{thm}\label{t:k-stab}
Let $\la,\mu,\nu\vdash n$ and suppose $\la_1 \le \mu_1 \le \nu_1$.  Denote $s=n-\nu_1$.
Let $k$ be fixed, such that
$$(\circledast) \qquad \min\{\la_k,\mu_k\} \, \geq \, \max\{\la_{k+1},\mu_{k+1} \} \. + \. s\..
$$
Then, for all \ts $t\geq 0$, we have:
$$g(\la,\mu,\nu) \. = \. g\bigl(\la + (t^k), \ts \mu +(t^k), \ts \nu + (tk)\bigr).
$$
\end{thm}

For $k=1$, we obtain the \emph{stability of Kronecker coefficients} which we state in this form:

\begin{cor}[\cite{PP_c}, Cor.~8.3]\label{c:1-stab}
Let $\la,\mu,\nu\vdash n$, \ts $\la_1 \le \mu_1 \le \nu_1$, and suppose  $\la_1+\nu_1 \le n + \max\{\la_2,\mu_2\}$.
Then, for all $t\ge 0$, we have:
$$g(\la,\mu,\nu) \. = \. g\left(\la + (t), \ts \mu +(t), \ts \nu + (t)\right) \. = \, \ov{g}(\al,\be,\ga)\ts,
$$
where $\al = (\la_2,\la_3,\ldots)$, $\be = (\mu_2,\mu_3,\ldots)$, $\ga = (\ga_2,\ga_3\ldots)$, and \.
$\ov{g}_1(\al,\be,\ga)$ \ts is the \emph{stable Kronecker coefficient}.
\end{cor}

\begin{proof}[Proof of Theorem~\ref{t:k-stab-limit}]
First, observe that condition~$(\circledast)$ in Theorem~\ref{t:k-stab}
holds for large enough~$t$.  This implies that $G(t)$ is eventually
constant.  Second, take the smallest $t_0$ such
that~$G(t_0)>0$.  Note that \ts $g\bigl(1^k,1^k,(k)\bigr) = 1$
by~\eqref{eq:kron-conj}.  Thus, by Manivel's inequality (Theorem~\ref{t:manivel}),
for all $t\ge t_0$ we have:
$$G(t+1) \, \ge \, \max \ts \bigl\{G(t), \ts
g\bigl(1^k,1^k,(k)\bigr)\bigr\}\, \ge \, \max\ts \bigl\{G(t), \ts 1 \bigr\}\ts,
$$
which proves that $G(t)$ is monotone increasing.
\end{proof}

\subsection{Proof of $k$-stability}
The proof is based on the technical, but powerful \emph{Reduction Lemma}
which we used in~\cite[Lemma~5.1]{PP_c} for complexity purposes.
Following~\cite{PP_c}, define the \emph{reduction map}
$$(\la,\mu,\nu) \. \to \. \bigl(\phi(\la),\phi(\mu),\phi(\nu)\bigr)
$$
as follows.\footnote{Here we are somewhat abusing notation, since the map $\phi$ is defined
on triples, rather than individual partitions.}

\smallskip

In the notation of the theorem, suppose $\ell(\la)$, $\ell(\mu)$, $\ell(\nu) \le \ell$
and $|\lambda_i -\mu_i| \leq s$ for all~$i\le \ell$ (otherwise, the map~$\phi$ is undefined).
Denote \ts $\omega = \lambda \cup \mu$ and \ts $\rho = \lambda \cap \mu$.
Let $I=\{ i: \rho_i \geq \omega_{i+1}+s, 1\leq i \leq \ell\}$, where
$\omega_{\ell+1}=0$. For all indices~$j$, set
$i_j = \min\{i \in I, i\geq j\}$.
Now let partitions \ts $\phi(\la)$ and \ts $\phi(\mu)$ \ts be defined by their parts:
$$\phi(\la)_j \. = \. \la_j - \rho_{i_j} +s(\ell+1-i_j)\., \ \
\phi(\mu)_j \. = \. \mu_j - \rho_{i_j} +s(\ell+1-i_j)\., \ \ \,
1 \le j \le \ell\ts.
$$
Let $r=|\phi(\la)| = |\phi(\mu)|$, where the latter equality follows
by construction.  Finally, define \ts $\phi(\nu) = (r-s, \nu_2,\nu_3,\ldots)\vdash r$.

\begin{lemma}[Reduction Lemma, \cite{PP_c}] \label{l:reduction}
Let $\lambda,\mu,\nu \vdash n$ and $\ell(\la), \ell(\mu),\ell(\nu) \leq \ell$.  Denote $s=n-\nu_1$. We have the following two cases:
\begin{enumerate}
\item[(i)] \ If $|\lambda_i - \mu_i| > s$ for some~$i$, then $g(\lambda,\mu,\nu)=0$,
\item[(ii)] \ If $|\lambda_i-\mu_i|\leq s$ for all~$i$, $1\le i \le \ell$, then
there is an integer \. $r \leq  2s\ell^2$, s.t.
$$g(\lambda,\mu,\nu) \. = \. g\bigl(\phi(\lambda),\phi(\mu),\phi(\nu)\bigr)
$$
for partitions \ts $\phi(\la),\phi(\mu),\phi(\nu)\vdash r$, defined as above.
\end{enumerate}
\end{lemma}

\smallskip

\begin{proof}[Proof of Theorem~\ref{t:k-stab}.]
 We apply the Reduction Lemma with $t$ as in the Theorem.
 Let $\ovl = \la+(t^k)$, $\ovm =\mu+(t^k)$ and $\ovn = \nu +(tk)$. Since $\lambda_i-\mu_i = \ovl_i -\ovm_i$ for all $i$, and $s\geq n-\nu_1 = (n+tk)-\ovn_1$, we conclude that triples of partitions \ts $(\lambda,\mu,\nu)$ and \ts $(\ovl,\ovm,\ovn)$ \ts fall in the same case of the Reduction Lemma.  In case~(i), both Kronecker coefficients from the Theorem are zero, and the equality trivially holds.

In case~(ii), we apply the reduction map~$\phi$ to both triples of partitions $(\la,\mu,\nu) $ and $\left(\ovl, \ovm, \ovn\right)$. Note that the condition in the Theorem implies that $k \in I$ for both triples, so in fact the sets $I$ are the same for both triples of partitions. We have
$\rho=\mu\cap \la$ and $$\ovr =\ovm\cap \ovl = (\mu+t^k)\cap (\la+t^k) = (\mu\cap\la)+(t^k) = \rho+(t^k).$$
Hence \ts $\ovl-\ovr = \la-\rho$ \ts and \ts $\ovm-\ovr=\mu-\rho$, so  for all~$i$ we have
$$\phi(\la)_j = \lambda_j - \rho_j +s(\ell+1-i_j)= (\al_j -\ovr_j) +s(\ell+1-i_j)= \phi(\ovl)_i.$$
Thus $\phi(\la)=\phi(\ovl)$ and similarly, $\phi(\mu)=\phi(\ovm)$. We also have $\phi(\nu) = (r-s,\nu_2,\ldots)= \phi(\ovn)$.
By case~(ii) of the Reduction Lemma, we have
$$
g(\la,\mu,\nu) \. = \. g(\phi(\la),\phi(\mu),\phi(\nu)) \. = \. g(\ovl,\ovm,\ovn)\ts,
$$
which completes the proof.\qedhere
\end{proof}

\begin{ex}\label{ex:k-stab-min-max}
{\rm
Let $\la =\mu = (2,2)$, $\nu=(3,1)$, $n=4$, and $k=1$.  Here $(n-\nu_1) = 1$, so
the min-max condition~$(\circledast)$ in Theorem~\ref{t:k-stab} does not hold.
A direct calculation gives $g(\la,\mu,\nu)=0$.
On the other hand,
$$g\bigl((2+t,2), \ts (2+t,2), \ts (3+t,1)\bigr) \, = \, 1 \qquad \text{for all} \ \ t\ge 1$$
(see e.g.~\cite{RW,Ros}).  This illustrates that shapes $\la$ and $\mu$ must be sufficiently
disconnected between rows $k$ and~$k+1$ for the theorem to hold.
}
\end{ex}

\subsection{Upper bounds}\label{ss:upper-stab-other}
The following result is a consequence of Theorem~\ref{t:k-stab}:

\begin{cor}  \label{c:upper-stab-max}
Let $\nu = (\nu_1,\ldots,\nu_\ell) \vdash s$ be fixed and $u=(\ell+1)\ell s$.
Then, for every $n\ge u$, partitions $\la,\mu \vdash n$ such that
$\ell(\la),\ell(\mu)\le \ell$, we have:
$$
g\bigl(\la,\ts \mu, \ts (n-s,\nu)\bigr) \, \le  \,  \max_{r\leq u} \. \max_{\al,\ts\be \vdash r, \ts \al \ssu \la, \ts \be \ssu \mu} \.
g\bigl(\al,\ts \be, \ts (r-s,\nu)\bigr).
$$
\end{cor}
\begin{proof}
In the notation of the Reduction Lemma, observe that if $\rho_k-\omega_{k+1}=t+2s$  for some  $k$ and $t\geq 0$ then we are either in case (i) of the Reduction Lemma and $g(\la,\mu,\nu)=0$, or else we can apply  Theorem~\ref{t:k-stab} to the partitions $\la-(t^k), \mu-(t^k), \nu-(tk)$. Thus we reduce the partitions to $\alpha,\beta,\gamma$ of size at most
$$\al_1+\ldots + \al_\ell \. \leq \. (\omega_1-\rho_2) + 2(\omega_2-\rho_3) + \ldots +\ell (\omega_\ell-0) \leq 2s \binom{\ell+1}{2} \. =\. \ell(\ell+1)s \. = \. u\..$$
This implies the result.
\end{proof}

Combining this corollary with Corollary~\ref{c:upper-min}, we obtain:

\begin{cor}  \label{c:upper-stab}
In conditions of Corollary~\ref{c:upper-stab-max}, we have:
$$
g\bigl(\la,\ts \mu, \ts (n-s,\nu)\bigr) \. \le  \. f^{(u-s,\nu)}\ts.
$$
\end{cor}

In other words, the coefficient in the corollary is bounded by a
constant independent of~$n$.  For example, for
$\ell(\la),\ell(\mu)\le \ell$, $\nu=(s)$, and fixed~$\ell$,
we have
$$
g\bigl(\la,\ts \mu, \ts (n-s,\nu)\bigr) \. \le  \. f^{(u-s,\nu)}\. <  \.
\binom{(\ell+1)^2s}{s} \. <  \. C(\ell)^s\..
$$

\subsection{Durfee square}\label{ss:upper-stab-durfee}
Denote by $\rd(\la)$ the Durfee square size:
$$\rd(\la) \. = \. \max\ts\{\ts r \ | \ \la_r\ge r\ts\}\ts.
$$
Here we use the $k$-stability for both usual and conjugate diagrams,
via the symmetry~\eqref{eq:kron-conj}, to obtain upper bounds for
partitions with small Durfee square (cf.~\cite{BR,Dvir}).

\begin{cor}\label{c:stab-durfee}
Let $\nu=(\nu_1,\ldots,\nu_\ell) \vdash s$ be fixed and $u=2\ts (h+1)^2s$. Then, for every $n\geq u$ and $\la,\mu \vdash n$,
such that $\rd(\la),\ts\rd(\mu) \leq h$, we have
$$
g\bigl(\la,\ts \mu, \ts (n-s,\nu)\bigr) \. \le  \.  \max_{r\leq u} \.\max_{\al,\ts\be \vdash r, \ts \al \ssu \la, \ts \be \ssu \mu} \.
g\bigl(\al,\ts \be, \ts (r-s,\nu)\bigr).
$$
\end{cor}
\begin{proof}
We use the notation of the Reduction Lemma and the definition of $\phi$. We can assume that $|\la_k-\mu_k|\leq s$, as otherwise the Kronecker coefficient is immediately 0 by case (i) of the Reduction Lemma. Suppose that there is an index $k$, such that $\omega_k-\rho_{k+1} =t+2s$ for some $t\geq 0$.  The partitions $\la-(t^k),\ts \mu- (t^k), \ts \nu-(tk)$ satisfy the conditions of Theorem~\ref{t:k-stab}, so
$$g\bigl(\la,\ts \mu, \ts (n-s,\nu)\bigr) = g\bigl(\la-(t^k),\ts \mu-(t^k), \ts (n-tk-s,\nu)\bigr).$$
We thus reduce the partitions to $(\al,\be,\ga)$ with the first $\ell$ parts of $\al$ and $\be$ adding up to at most
\ts $h^2+ 2s+4s+\ldots +2h s = h^2 +h(h+1)s$.
Recall the symmetry~\eqref{eq:kron-conj}:
$$
g(\al,\be,\ga) \. = \. g(\al',\be', \ga)\ts.
$$
Now apply the same reasoning for the partitions $\al',\be',\ga$ to reduce their first~$h$ parts,
and conjugate them again.  The resulting partitions have size at most $h^2 + 2 h(h+1)s\le u$ then,
as desired
\end{proof}

\begin{cor}  \label{c:upper-stab-durfee}
In conditions of Corollary~\ref{c:stab-durfee}, we have:
$$
g\bigl(\la,\ts \mu, \ts (n-s,\nu)\bigr) \. \le  \. f^{(u-s,\nu)}\ts.
$$
\end{cor}

Except for a weaker bound on~$u$, Corollary~\ref{c:upper-stab-durfee}
is an extension of Corollary~\ref{c:upper-stab}.  For example, for
$\rd(\la),\rd(\mu)\le h$, $\nu=(s)$, and fixed~$h$,
we have
$$
g\bigl(\la,\ts \mu, \ts (n-s,\nu)\bigr) \. \le  \. f^{(u-s,\nu)}\. <  \.
\binom{2\ts(h+1)^2s}{s} \. <  \. C(h)^s\..
$$
Note that for $h=1$, partitions~$\la$ and~$\mu$ are hooks and their
Kronecker coefficients are well understood (see~\cite{Rem,Ros}).
Curiously, the upper bound $3$ in~~\cite[Cor.~17]{Ros}
is independent on~$s$ in that case.

\bigskip

\section{Bounds via characters}\label{s:char}

\subsection{Proof of Theorem~\ref{t:char_effective}}
Denote by $\chi\dwn$ the restriction of the $S_n$-representation $\chi$
to~$A_n$, and by $\psi\upa$ the induced $S_n$-representation of the
$A_n$-representation~$\psi$.  We refer to~\cite[$\S$2.5]{JK} for basic results in representation theory of~$A_n$.
Recall that if $\nu\neq \nu'$, then $\chi^\nu\dwn = \chi^{\nu'}\dwn=\psi^\nu$ is irreducible in~$A_n$.
Similarly, if $\nu=\nu'$, then $\chi^\nu\dwn = \psi^\nu_+ \oplus \psi^\nu_-$, where $\psi^\nu_{\pm}$
are irreducible in~$A_n$, and are related via $\psi^\nu_+[ (12)\pi(12)] = \psi^\nu_-[\pi]$.

Consider now the conjugacy classes of~$A_n$ and the corresponding character values.
Denote by $C^{\alpha}$ the conjugacy class of $S_n$ of permutations of cycle type~$\alpha$, and
by $\cD\ssu \cP$ the set of partitions into distinct odd parts.
We have two cases:

\smallskip

\nin
{\small $\mathbf{(1)}$} \ For $\alpha\notin\cD$, we have $C^\alpha$ is also a conjugacy class of~$A_n$.  Then
$$
\aligned
\chi^\nu\dwn[C^{\alpha}] \,  = \, \chi^\nu[C^\alpha] \quad & \text{if} \ \ \nu \neq \nu'\ts,\\
\psi^\nu_{\pm}[C^\alpha] \,  = \, \frac12 \. \chi^\nu[C^\alpha]\quad & \text{if} \ \ \nu=\nu'\ts.
\endaligned
$$

\nin
{\small $\mathbf{(2)}$} \ For $\alpha\in\cD$, we have $C^\alpha = C^{\alpha}_+ \cup C^{\alpha}_-$, where $C^{\alpha}_{\pm}$ are conjugacy classes of~$A_n$. Then
$$
\aligned
\chi^\nu \dwn [C^{\alpha}_{\pm}] \,  = \, \chi^\nu [C^\alpha] \quad & \text{if} \ \ \nu \neq \nu'\ts,\\
\psi^{\nu}_{\pm} [C^\alpha_{\pm}] \, = \, \frac12 \chi^\nu[C^\alpha] \quad & \text{if} \ \ \nu = \nu' \ \ \. \text{and} \ \ \. \alpha \neq \wh\nu\ts, \\
\psi^{\nu}_{\pm}[C^{\wht{\nu}}_+] - \psi^{\nu}_{\pm}[C^{\wht{\nu}}_-] \,  = \, \pm e_{\nu}  \quad & \text{if} \ \ \nu=\nu' \ \ \. \text{and} \ \ \.
e_{\nu} = (\wh\nu_1\ts \wh\nu_2 \ts \cdots\ts)^{1/2}>0\ts.
\endaligned
$$

\nin
Now, by the Frobenius reciprocity, for every \ts $\mu=\mu'$ \ts we have:
$$
\langle \psi^{\mu}_{\pm}\upa , \chi^{\alpha}\rangle =  \langle \psi^{\mu}_\pm , \chi^\alpha\dwn\rangle\ts,
$$
which is nonzero exactly when $\alpha=\mu$ and so $\psi^{\mu}_{\pm}\upa = \chi^\mu$.  This implies
\begin{equation}\label{eq:kron}
\aligned
g(\la,\mu,\mu) \, & = \, \langle \chi^\mu\otimes \chi^\la, \chi^\mu \rangle \, = \,
\langle \chi^\mu\otimes \chi^\la, \ts \psi^\mu_{\pm}\upa \rangle \, = \,
\bigl\langle (\chi^\mu\otimes \chi^\la)\dwn, \psi^\mu_{\pm} \bigr\rangle \\
& = \,
\bigl\langle \psi^\mu_+ \ts \otimes \ts \chi^\la\dwn, \ts \psi^\mu_{\pm} \bigr\rangle \. + \.
\bigl\langle \psi^\mu_-\ts \otimes \ts \chi^\la\dwn, \ts \psi^\mu_{\pm} \bigr\rangle\ts.
\endaligned
\end{equation}

\smallskip
We can now estimate the Kronecker coefficient in the theorem.
First, decompose the following tensor product of the $A_n$ representations:
\begin{equation}\label{eq:decomp}
 \psi^{\mu}_+ \ts \otimes \ts  \chi^\la\dwn \, = \, \oplus_{\tau} \. m_\tau \ts \psi^\tau\ts,
 \end{equation}
where $\psi^\tau$ are all the irreducible representations of~$A_n$, the coefficients~$m_\tau$
are their multiplicities in the above tensor product, and $\tau$ goes over the appropriate indexing.

Note that for any character $\chi$ of $S_n$ and $\pi \in A_n$ we trivially have $\chi\dwn [\pi] =\chi[\pi]$.
% , where $\pi$ is any permutation from the given conjugacy class.
Evaluating that tensor product on the classes $C^{\wht\mu}_{\pm}$ gives
$$
\bigl(\psi^\mu_+\otimes \chi^\la\dwn\bigr)\bigl[ C^{\wht{\mu}}_+\bigr] \. - \. \bigl(\psi^\mu_+\otimes \chi^\la\dwn\bigr)\bigl[ C^{\wht{\mu}}_-\bigr] \,
= \,
\chi^\la\dwn\bigl[C^{\wht{\mu}}_{\pm}\bigr] \ts \Bigl( \psi^\mu_+\bigl[C^{\wht\mu}_+\bigr] \. - \. \psi^\mu_+\bigl[C^{\wht\mu}_-\bigr] \Bigr) \,
=\, \chi^\la\bigl[C^{\wht{\mu}}\bigr] \ts e_{\nu} \..
$$
On the other hand, evaluating the right-hand side of equation~\eqref{eq:decomp} gives
$$\aligned
 &\qquad \quad  \bigl(\psi^\mu_+\otimes \chi^\la\dwn \bigr)\bigl[ C^{\wht{\mu}}_+\bigr] \. - \. \bigl(\psi^\mu_+\otimes \chi^\la\dwn\bigr)\bigl[ C^{\wht{\mu}}_-\bigr]
 \, = \,
\sum_\tau \. m_\tau \Bigl(\psi^i\bigl[C^{\wht\mu}_+\bigr] - \psi^i\bigl[C^{\wht\mu}_-\bigr]\Bigr) \\
& = \, m_{\mu+}\Bigl(\psi^{\nu}_{+}\bigl[C^{\wht{\nu}}_+\bigr] - \psi^{\nu}_{+}\bigl[C^{\wht{\nu}}_-\bigr]\Bigr)\.
+ \.m_{\mu-}\Bigl(\psi^{\nu}_{-}\bigl[C^{\wht{\nu}}_+\bigr] - \psi^{\nu}_{-}\bigl[C^{\wht{\nu}}_-\bigr]\Bigr) \,
= \, \bigl(m_{\mu+}-m_{\mu-}\bigr) \ts e_{\nu}.
\endaligned
$$
Here we used the fact that all characters are equal at the two classes $C^{\wht{\mu}}_{\pm}$, except for the ones corresponding to~$\mu$.
Equating the evaluations and using $e_\nu>0$, we obtain
$$
m_{\mu+}\ts - \ts m_{\mu-} \. = \. \chi^\la\bigl[C^{\wht{\mu}}\bigr]\ts.
$$
This immediately implies
\begin{equation}\label{eq:max-char}
\max \bigl\{m_{\mu+},m_{\mu-}\bigr\} \, \geq \, \Bigl| \chi^\la\bigl[C^{\wht{\mu}}\bigr] \Bigr|
\end{equation}

On the other hand, since all inner products are nonnegative, the equation~\eqref{eq:kron} gives
$$
g(\la,\mu,\mu) \, \geq \, 
\max\left\{\langle \psi^\mu_+\otimes \chi^\la\dwn, \. \psi^\mu_{+} \rangle, 
\langle \psi^\mu_+\otimes \chi^\la\dwn, \psi^\mu_{-} \rangle\right\} \, =\, 
\max \bigl\{m_{\mu+},m_{\mu-}\bigr\}\.,
$$
and now equation~\eqref{eq:max-char} implies the result. \ $\sq$

\subsection{Stanley's theorem} \label{ss:stanley}

We give a new proof of the following technical result
by Stanley \cite[Prop.~11]{Sta-unim}.  Our proof uses Theorem~\ref{t:char_effective}
and Almkvist's Theorem~\ref{t:almkvist-unimod}.  Both results are
crucially used in the next section.

\begin{thm}[Stanley] \label{t:stanley-unimod}
The following polynomial in $q$ is symmetric and unimodal
$$
\binom{2n}{n}_q \, - \, \. \prod_{i=1}^{n}\. \ts \bigl(1+q^{2i-1}\bigr)\..
$$
\end{thm}

\begin{proof}
Let $\mu=(n^n)$ and $\tau_k=(n^2-k,k)$, where $k\le n^2/2$.
By the two coefficients lemma (Lemma~\ref{l:g_partitions}),
we have
$$g(\la,\mu,\mu) \, = \, p_k(n,n) \. - \. p_{k-1}(n,n)\ts.
$$
By the Jacobi-Trudi identity %Frobenius character formula  --- this is something else according to the literature I find (Etingof), so let's not use that name, ok?
and the  Murnaghan--Nakayama rule, we have:
$$\chi^{\tau_k}[\wht{\mu}] \, = \, \chi^{(n^2-k) \circ (k) }\bigl[\wht{\mu}\bigr] \.
- \. \chi^{(n^2-k+1) \circ (k-1)}\bigl[\wht{\mu}\bigr] \,
=\, b_k(n) - b_{k-1}(n)\ts.
$$
(cf.~\cite{PP,PPV}).
Applying Theorem~\ref{t:char_effective} with $\la=\tau_k$ and $\mu$ as above, we have:
$$p_k(n,n)\. -\. p_{k-1}(n,n) \, = \, g(\la,\mu,\mu) \,
\geq \,\bigl|\chi^{\la}\bigl[\wht{\mu}\bigr] \bigr| \, = \, b_k(n)\. - \. b_{k-1}(n)\ts.
$$
The last equality follows from Almkvist's Theorem~\ref{t:almkvist-unimod}.
Reordering the terms, we conclude
$$p_k(n,n) \. - \. b_{k}(n) \, \geq \, p_{k-1}(n,n) - b_{k-1}(n)\ts,
$$
which implies unimodality.  The symmetry is straightforward.
\end{proof}

\subsection{Asymptotic applications} Let $\rho_m=(m,m-1,\ldots,2,1)$ be the
\emph{staircase shape}, $n=|\tau_m|=\binom{m+1}{2}$. The coefficient
$g(\rho_m,\rho_m,\nu)$ first appeared in connection with the
\emph{Saxl conjecture}~\cite{PPV}, and was further studied
in~\cite[$\S$8]{Val2}.

For simplicity, let $m=1$~mod~4, so~$n$ is even and
\ts $\wh \rho_m = (1,5,\ldots,2m-1)$.  Let $\tau_k=(n-k,k)$.
Applying Theorem~\ref{t:char_effective} and
the Murnaghan--Nakayama rule as above,
we have
$$g\bigl(\rho_m,\rho_m,\tau_k) \, \ge \,
\bigl|\chi^{\tau_k}\bigl[\wh \rho_m\bigr]\ts\bigr| \, = \, \parti_R(k) \. - \.
\parti_R(k-1)\ts,$$
where $\parti_R(k)$ is the number of partitions of~$k$ into
$R=\{1,5,\ldots,2m-1\}$.

In the ``small case'' \ts $k \le 2m$, by the Roth--Szekeres theorem~\cite{RS},
we have:
$$g\bigl(\rho_m,\rho_m,\tau_k) \, \ge \, \parti_R(k) \. - \.
\parti_R(k-1) \, \sim \, \frac{\pi\ts\sqrt{2}}{3\ts k^{3/2}} \.
e^{\pi\sqrt{k/6}}\.,
$$
i.e.~independent of~$n$.
On the other hand, the upper bond in $\S$\ref{ss:gen-dim} gives
$$
g\bigl(\rho_m,\rho_m,\tau_k) \, \le \, f^{\tau_k} \, < \, \frac{n^k}{k!}\ts,
$$
leaving a substantial gap between the upper and lower bounds.
For $k=O(1)$ bounded, Theorem~8.10 in~\cite{Val2}, gives
$$
g\bigl(\rho_m,\rho_m,\tau_k) \, \sim \, m^k \, \sim \, (2\ts n)^{k/2}\quad
\text{as} \ \
n\to \infty \ts,
$$
suggesting that the upper bound is closer to the truth.  In fact, the proof
in~\cite{Val2} seems to hold for all $k=o(m)$.

In the ``large case'' \ts $k =n/2\sim m^2/4$, the Odlyzko--Richmond result (\cite[Thm.~3]{OR}) gives
$$g\bigl(\rho_m,\rho_m,\tau_k) \, \ge \,
\parti_R(k) \. - \.
\parti_R(k-1) \, \sim \, \frac{3^{3/2}}{2^{15/4}\ts \sqrt{\pi} \ts m^3} \. 2^{m/4}
\, \sim \, \frac{3^{3/2}}{2^{47/4}\ts \sqrt{\pi} \ts k^{3/2}} \. \. 2^{\sqrt{k}/2}
\ts.
$$
For the upper bound, equation~\eqref{eq:upper-cat} gives
$$g\bigl(\rho_m,\rho_m,\tau_k) \, \le \, f^{\tau_k} \, \lesssim \, \frac{1}{\sqrt{\pi} \ts k^{3/2}} \, 4^k\ts.
$$

\bigskip

\section{Restricted partitions}\label{s:qbin}

\subsection{Analytic estimates} \label{ss:qbin-alm}
The proof of Almkvist's Theorem~\ref{t:almkvist-unimod} is based on the
following technical results.

\begin{lemma}[\cite{A1}]  \label{l:alm1}
For $3\leq k \leq 2n+1$, we have:
$$
b_k(n)\ts - \ts b_{k-1}(n) \, = \, b_k(n-1) \ts - \ts b_{k-1}(n-1)\ts.
$$
Similarly, for $2n+2\leq k\leq (n-1)^2/2$, we have:
$$
b_k(n)\ts -\ts b_{k-1}(n) \, = \, b_k(n-1)\ts -\ts b_{k-1}(n-1) \ts +\ts b_{k-2n+1}(n-1)\ts -\ts b_{k-2n}(n-1)\ts.
$$
\end{lemma}

\smallskip

\begin{lemma}[\cite{A1}]  \label{l:alm2}
For $n\geq 83$ and $(n-1)^2/2\leq k \leq n^2/2$, we have:
$$
b_k(n) \ts -\ts b_{k-1}(n) \, \geq \, C \. \frac{2^n}{n^{9/2} }\,,
\quad \text{where} \quad C\ts = \. \frac{3\ts\sqrt{3}}{\sqrt{2} \ts \pi^2} \. \approx \ts 0.37\..
$$
\end{lemma}

Based on this setup, we refine Almkvist's Theorem~\ref{t:almkvist-unimod} as follows

\begin{thm}\label{t:almkvist_eff}
For any $n\geq 31$, and $26 \leq k \leq  n^2/2$ we have:
$$
b_k(n)\. -\. b_{k-1}(n) \, \geq \, C \.  2^{\sqrt{2\ts k}} \ts \frac{ 1}{(2\ts k)^{9/4} }\..
$$
\end{thm}

\begin{proof}
Denote
$$\da_k(n) \, = \, b_k(n)\. -\. b_{k-1}(n)\ts.
$$
First, let \ts $n\geq 83$ and \ts $(n-1)^2/2 \leq k \leq n^2/2$.  By Lemma~\ref{l:alm2}, we have
\begin{equation}\label{eq:aml-mid}
\da_k(n) \, \geq \, C \. 2^n \. \frac{ 1 }{n^{9/2} } \, \geq \, C \. 2^{\sqrt{2\ts k}} \. \frac{ 1}{(2j)^{9/4} }\,,
\end{equation}
where the last inequality follows since the function \ts $f(x)=\log2 \sqrt{x} -9/4 \log (x)$ \ts
is increasing.

The recurrence relations in Lemma~\ref{l:alm1} and Almkvist's theorem $\vt_k(n)\ge 0$ give
$$
\da_k(n) \, \geq \, \da_k(n-1) \quad \text{for all} \quad  3 \le k \le n^2/2 \ts , \; k\neq 2n+1 \ts
$$
and $\da_{2n+1}(n) \, = \, \da_{2n+1}(n-1)-1 \, =\, \da_{2n+1}(n-4)$.
Now, let $r$ be such that $(n-r-1)^2/2 \leq k \leq (n-r)^2/2$, and $n-r \geq 83$.
Applying~\eqref{eq:aml-mid} to $(n-r)$, we conclude:
$$\da_k(n)\, \geq \, \da_k(n-r) \, \geq \, C \. 2^{\sqrt{2k}} \. \frac{ 1}{(2k)^{9/4} }\,.
$$
Next, we check by computer that the inequality in the Theorem holds for all $n \in \{31,\ldots,83\}$ and $26 \leq k \leq n^2/2$. Finally, for $k\leq 83^2/2$ and $n > 83$, we apply the inequalities of Lemma~\ref{l:alm1} repeatedly to obtain
$$\da_k(n)\, \geq \, \da_k(83) \, \geq \, C \. 2^{\sqrt{2k}} \. \frac{ 1}{(2k)^{9/4} }\,. \hfill \qedhere$$
\end{proof}

\begin{cor}\label{c:square_bin}
Let $n \geq 8$, $1 \le k \le n^2/2$, $\mu=(n^n)$ and $\tau_k=(n-k, k)$.  Then
$$
g\bigl(\mu,\mu,\tau_k \bigr)\, \geq \, C \. \frac{ 2^{\sqrt{2\ts k}} }{(2\ts k)^{9/4} }\,,
\quad \text{where} \quad C\ts = \. \frac{\sqrt{27/8}}{\pi^2}\..
$$
\end{cor}
\begin{proof}
Following the proof of Stanley's Theorem~\ref{t:stanley-unimod}, for all $26\leq k\leq n^2/2$ and $n\geq 31$
Theorem~\ref{t:almkvist_eff} gives:
$$
g\bigl(\mu,\mu,\tau_k \bigr)\, = \, p_k(n,n) - p_{k-1}(n,n) \,\geq
\, b_k(n)-b_{k-1}(n) \, \geq \, C \. \frac{ 2^{\sqrt{2\ts k}} }{(2\ts k)^{9/4} }\. .
$$
For the remaining values of $n$ and $k$ we check the inequality by computer, noticing that $p_k(n,n) =p_k(26,26)$ when $k\leq 26$.
\end{proof}

%\bigskip

\subsection{Partitions in rectangles} \label{ss:qbin-rect}
%
% Denote \ts $\pq_k(\ell,m) \ts = \ts p_k(\ell,m) -  p_{k-1}(\ell,m)$.
By Lemma~\ref{l:g_partitions}, we have
$$\pq_k(\ell,m) \, := \, p_k(\ell,m) \. - \. p_{k-1}(\ell,m) \, = \, g(m^\ell,m^\ell,\tau_k)\ts.
$$

\begin{thm}\label{t:qbin-rect}
Let $8\le \ell \le m$ and $1\le k \le m\ell/2$.
Define $n$ as $$n = \begin{cases} \lfloor 2\frac{\ell-8}{2} \rfloor ,& \text{ when $\ell m$ is even, }\\
2 \lfloor \frac{\ell-8}{2}\rfloor -1, & \text{ when $\ell m$ is odd, }
\end{cases}$$
and let $v=\min(k, n^2/2)$. Then:
$$
\pq_k(\ell,m) \, \geq \, C \frac{ 2^{\sqrt{v}} }{v^{9/4} } \qquad
\text{where} \quad C \. = \. \frac{3\sqrt{3}}{\sqrt{2}\.\pi^2}\..
$$
\end{thm}

\begin{proof}
We apply Theorem~\ref{t:manivel} to bound the Kronecker coefficient for rectangles with an appropriate Kronecker coefficient for a square and then apply Corollary~\ref{c:square_bin}.

By strict unimodality~\eqref{eq:strict-unim}, we have that $g(m^\ell,m^\ell, (m\ell-k,k) )>0$ for all $\ell,m \geq 8$.
By Corollary~\ref{c:square_bin}, we can assume $\ell <m$. Assume first that $\ell >16$.

First, suppose that $\ell m$ is even and let $n= 2\lfloor \frac{\ell-8}{2} \rfloor $. Then for any $1<k \leq \frac{\ell m}{2}$ we can find $1\neq r\leq \frac{ (m-n)\ell}{2}$ and $1\neq s \leq \frac{ n\ell}{2} $, such that $k=r+s$. Take $s=\min(k,n\ell/2)$ . Let $\tau_k = (m\ell-k,k)$ and $\tau_r = \left( (m-n)\ell-r,r \right)$, $\tau_s= \left( n\ell-s,s\right)$. Apply Theorem~\ref{t:manivel} to the triples $\left( (m-n)^\ell, (m-n)^\ell, \tau_r \right)$ and $\left( n^\ell, n^\ell, \tau_s\right)$ to obtain

$$\pq_k(\ell,m) \, = g( m^\ell, m^\ell, \tau_k) \geq \max\left( g\bigl( (m-n)^\ell, (m-n)^\ell, \tau_r \bigr), g\bigl( n^\ell, n^\ell, \tau_s\bigr) \right)\, \geq \, \pq_s(\ell,n).$$
Similarly, dividing the $n \times \ell$ rectangle into $n\times n$ square and $n \times (n-\ell)$ rectangle, where $n\ell$ is again even, we have
$$\pq_s(\ell,n) \geq \pq_{s'}(n,n),$$
where $s'=\min(s, n^2/2)=\min(k, n^2/2)$.

In the case that both $\ell$ and $m$ are odd, the only case where the above reasoning fails is when $k=\lfloor m\ell /2 \rfloor$ and $r,s$ don't exist. In this case we take $n=2 \lfloor \frac{\ell-8}{2}\rfloor -1$ and we can always find $r,s$. In summary, we have that
$$\pq_k(\ell,m) \geq \pq_{v}(n,n),$$
where
$$n = \begin{cases} 2 \lfloor \frac{\ell-8}{2} \rfloor ,& \text{ when $\ell m$ is even }\\
2 \lfloor \frac{\ell-8}{2}\rfloor -1, & \text{ when $\ell m$ is odd }
\end{cases}$$
and
$v = \min(k, n^2/2)$.
Now apply Corollary~\ref{c:square_bin} to bound $\pq_{v}(n,n)$ and obtain the result for $\ell >16$.

When $\ell \leq 16$, and $m \geq 24$, we can apply the same reasoning as above to show $\pq_k(\ell,m) \geq \pq_{v}(\ell,16)$. Then for $\ell,m\leq 16$ the statement is easily verified by direct calculation.
\end{proof}

\begin{proof}[Proof of Theorem~\ref{t:qbin-main}]
For $n\geq \ell-9$, the desired inequality then follows from Theorem~\ref{t:qbin-rect} and the observation that
$$\frac{2^{n/\sqrt{2}} }{ n^{9/2} }  \, \geq  \, 2^{-9/\sqrt{2} } \. \frac{2^{\ell/\sqrt{2}} }{ \ell^{9/2} }\..
$$
Taking \ts $A = 2^{-9/\sqrt{2} }  \ts C \approx 0.00449$ \ts gives the desired inequality for all values.
\end{proof}

\smallskip

\subsection{Upper bounds}  \label{ss:qbin-upper}
Let $k\le \ell\le m$ and $n=\ell\ts m$. We have:
$$
\pq_k(\ell,m) \. = \. p_k(\ell,m) \ts - \ts p_{k-1}(\ell,m) \.
= \. \parti(k) \ts - \ts \. \parti(k-1) \. = \. \parti'(k) \.
\sim \. \frac{\pi}{12\ts\sqrt{2}\ts k^{3/2}} \,\, e^{\pi \sqrt{\frac{2}{3}\ts k}}
$$
Compare this with the lower bound in Theorem~\ref{t:qbin-main}:
$$\pq_k(\ell,m) \, > \, A \, \frac{2^{\sqrt{2\ts k}}}{(2\ts k)^{9/4}}\,.
$$
There is only room to improve the base of exponent here:
$$
\text{from} \ \ 2^{\sqrt{2}} \.\approx \. 2.26 \ \   \text{to}
\ \ e^{\pi \sqrt{\frac{2}{3}}} \. \approx \. 13.00\,.
$$
In fact, using our methods, the best lower bound we can hope
to obtain is
$$
e^{\pi \sqrt{\frac{1}{6}}} \. \approx \. 3.61\,,
$$
which is the base of exponent in the Roth--Szekeres formula
for the number $b_k(n)$ of unrestricted partitions into distinct
odd parts, where~$n\ge k$.

\smallskip

For a different extreme, let $m=\ell$ be even, and $k=m^2/2$.
We have the following sharp upper bound:
$$
\pq_k(m,m) \. \le \. p_k(m,m) \. \sim \. \sqrt{\frac{3}{\pi \ts m}}
\ts  \binom{2\ts m}{m} \. \sim \. \frac{\sqrt{3}\, 4^m}{\pi \ts m}\,.
$$
On the other hand, the lower bound in Theorem~\ref{t:qbin-main} gives:
$$\pq_k(m,m) \, > \, A \, \frac{2^{m}}{m^{9/2}}\,.
$$
Again, we cannot improve the base of the exponent~$2$ with our method,
simply because the total number of partitions into distinct odd
parts~$\le 2m-1$, is equal to~$2^m$.

\bigskip

{\small

\section{Final remarks} \label{s:fin}

\subsection{} \label{ss:fin-GCT}
It is rather easy to justify the importance of the Kronecker coefficients
in Combinatorics and Representation Theory.  Stanley writes:
``One of the main problems in the combinatorial representation theory
of the symmetric group is to obtain a combinatorial interpretation for
the Kronecker coefficients''~\cite{Sta}.

The Geometric Complexity Theory (GCT) is a more recent interdisciplinary
area, where computing the Kronecker coefficients is crucial (see~\cite{MS}).
B\"urgisser voices a common complaint of the experts:
``frustratingly little is known about them''~\cite{Bur}.
We refer to~\cite{PP_c,PP-future} for details and further references.

Part of this work is motivated by questions in GCT.  Specifically, the
experts seem to be interested in estimating the coefficients
$$
g\bigl(m^\ell,m^\ell, \la)\ts, \quad \text{where} \ \ \ \la \vdash \ell\ts m\ts.
$$
Several bounds in this paper are directly applicable to this case and
we plan to return to this problem in the future.

\subsection{} \label{ss:fin-stab}
The stable (reduced) Kronecker coefficients go back to the early papers by
Murnaghan (1938, 1955).  More recently, an effort was made to determine
the smallest $t$ the equality in Theorem~\ref{t:k-stab-limit} holds for~$k=1$,
see~\cite{BOR2,Val1}.  The bound in Corollary~\ref{c:1-stab}
is roughly of the same order as the best bounds, but weaker in
some cases (cf.~\cite[$\S$8.10]{PP_c}).  In fact, by looking at two-row
partitions, one can show that bounds in Theorem~\ref{t:k-stab} cannot
be substantially improved (cf.~Example~\ref{ex:k-stab-min-max}).

The monotonic increase of $G(t)$ in Theorem~\ref{t:k-stab-limit} generalizes
Brion's theorem for~$k=1$, see~\cite{Bri}. This work was motivated
by the classical \emph{Foulkes conjecture} giving inequality for certain
plethystic coefficients (see e.g.~\cite{Ves}).

\subsection{} \label{ss:fin-LR}
It was shown by Murnaghan (1955) and Littlewood (1958) that the
LR~coefficients are the stable Kronecker coefficients in a special case:
\begin{equation}\label{eq:LR-Kron}
c^\la_{\mu\nu} \. = \. \ov{g}_1\bigl( (t-|\la|,\la),
\ts (t-|\mu|,\mu), \ts (t-|\nu|,\nu)\bigr)\ts,
\end{equation}
for $t-|\la| \ge \la_1$, $t-|\mu| \ge \mu_1$, and $t-|\nu| \ge \nu_1$.
It would be instructive to compare bounds for the LR~coefficients with
our general bounds on Kronecker coefficients.

In recent years, a lot of effort has been made to prove inequalities
for the LR~coefficients in many special cases.  Notably, Okounkov's
(now disproved) conjecture on log-concavity of
LR~coefficients has been a source of
inspiration (see~\cite{CDW-log,Nar2,Oko}).  Also, various
\emph{Schur positivity} results and conjectures can be written as
inequalities for certain LR~coefficients (see e.g.~\cite{BBR,LPP}).

\subsection{}  \label{ss:fin-cont}
Recall that LR-coefficients are \SP-complete, and Kronecker coefficients
are \SP-hard~\cite{Nar1}.  This follows from~\eqref{eq:LR-Kron} and
effective bounds discussed in $\S$\ref{ss:fin-stab} (see also~\cite{BI}).
However, some \SP-complete problems have polynomial approximation
schemes (FPRAS) using Monte Carlo methods, or otherwise.
Notably, the knapsack problem which~\cite{Nar1} reduced to computing
LR-coefficients is known to be have a FPRAS (see~\cite{PP-future}).

To appreciate the complexity of approximating LR and Kronecker
coefficients, recall a simple reduction from general contingency
tables to LR~coefficients given in~\cite{PV2}.
Approximating the number $\CT(\al,\be)$ of $\ell \times m$ contingency
tables with given marginal sums~$\al$ and~$\be$ is a notoriously
difficult problem both theoretically and algorithmically. We refer
to~\cite{DG} for a broad survey, and~\cite{PP-future} for
more recent references.

Let us mention here the recent estimate for $\CC(\cdot)$ in the
case of  smooth boundary conditions~\cite{Ben}.  For \emph{all} margins,
the only upper bound we know is the elegant but hard to compute
bound in~\cite{Bar1}.  Thus, we included the Proposition as a simple
benchmark measure to compare with other upper bounds.

\subsection{}  \label{ss:fin-char}
In a special case $\la=\mu=\mu'$, Theorem~\ref{t:char_effective} and
the Murnaghan--Nakayama rule gives a weak bound $g(\mu,\mu,\mu) \ge 1$
proved earlier in~\cite{BB}.  In~\cite{PPV}, we apply a qualitative
version of this result to a variety of partitions generalizing
hooks and two-row partitions.
Unfortunately, computing the characters of $S_n$ in
\SP-hard~\cite{PP_c}.  It is thus unlikely that this
approach can give good bounds for general
\emph{tensor squares} $g(\la,\mu,\mu)$ (cf.~\cite{Val2}).

\subsection{} \label{ss:fin-two-rows}
For general triples, neither of the three upper
bounds on $g(\la,\mu,\nu)$ given Section~\ref{s:gen}, seem to
dominate another.  We use the example $g(\la,\la,\la)$ with
$\la=(m,m)$  for simplicity and clarity of exposition. The exact
value in this case is~$0$ for odd~$m$ and~$1$ for even~$m$ (see~\cite{BOR}).
This underscores the difficulty of getting sharp lower and upper bounds.

%\subsection{Unimod}  \label{ss:fin-unimod}
%Stanley's theorem...

\subsection{}  \label{ss:fin-qbin}
As we mentioned in the introduction, the first lower bound
$\pq_k(\ell,m)\ge 1$ was obtained by the authors in~\cite{PP_s}.
This was quickly extended in followup papers~\cite{Dha} and~\cite{Zan},
both of which employing O'Hara's combinatorial proof~\cite{O}.  First,
Zanello's proof gives:
$$
\pq_k(\ell,m) \. > \. d\., \ \ \ \text{for}
\ \ \ell \. \ge \. d^2+5d+12\ts, \ \, m \. \ge \. 2d+4\ts, \ \,
4d^2 + 10d+7 \le \. k \. \le \ell m/2\ts,
$$
see the proof of Prop.~4 in~\cite{Zan}.  Similarly, Dhand's proves
somewhat stronger bounds
$$
\pq_k(\ell,m) \. > \. d\., \ \ \ \text{for} \ \
\ell, \ts m \. \ge \. 8\ts d, \ \ 2 d \le \. k \. \le \ell m/2\ts,
$$
see Theorem~1.1 in~\cite{Dha}.  Ignoring constraints of $\ell$ and~$m$,
the bounds give $\pq_k =\Omega(\sqrt{k})$ and $\pq_k =\Omega(k)$, which
are substantially inferior our $\pq_k =\exp \Omega(\sqrt{k})$ bound
in Theorem~\ref{t:qbin-main}.

The sequence $p_k(\ell,m)$ has remarkably sharp asymptotics bounds,
including the \emph{central limit theorem} (CLT) in much greater
generality~\cite{CJZ} (see also references in the Corrigendum),
the Hardy--Ramanujan type formula~\cite{AA}, and especially
the CLT with the error bound~\cite{Tak}.  Also, when
$\ell$ is fixed, sharp asymptotic bounds are given in~\cite{SZ}.

%\subsection{}  \label{ss:fin-qbin-clt}
%

\subsection{}  \label{ss:fin-qbin-ext}
One is tempted to ask whether for Theorem~\ref{t:k-stab-limit} can be
extended to general $\la,\mu,\nu \vdash n$ and $\al,\be,\ga\vdash k$,
to give a bounded function
$$
G(t) \, := \, g(\la + t\al, \mu + t\be, \nu + t\ga)\ts.
$$
Unfortunately, this fails already for $\la=\mu=\nu=\emp$,
$\al=\be=(1^4)$ and $\ga=(2^2)$.  Using Lemma~\ref{l:g_partitions}, we have:
$$
G(t) \, = \, g\bigl(t^4,\ts t^4,\ts t^2\bigr) \, = \, p_{2t}(4,t) \. - \. p_{2t-1}(4,t) \, = \, \theta(t)\ts,
$$
where the asymptotics follows from~\cite{West} or~\cite[Th.~4.6]{SZ}.

\subsection{}\label{ss:gen-LR}
Curiously, a bound similar to propositions~\ref{p:upper-dim} and~\ref{p:upper-GL},
holds for the LR~coefficients as well.  Although we do not need it for this work, we
include it here for reader's convenience.

\begin{prop}  \label{p:upper-lr}
Let $\al\vdash p+q$, $\be\vdash p$, $\ga \vdash q$.
Then, the LR coefficient satisfies:
$$c^{\al}_{\beta\gamma} \, \leq \. \frac{f^{\al}}{f^{\beta}\ts f^{\gamma}}\,,$$
\end{prop}

\begin{proof}
Recall the identity
\begin{equation}\label{eq:lr-schur}
s_{\al}(x,y) \, = \, \sum_{\eta,\ts \zeta} \. c^{\al}_{\eta\zeta} \. s_{\eta}(x)\ts s_{\zeta}(y)\.,
\end{equation}
where the summation is over all $\eta, \zeta \in \cP$, such that \ts
$|\eta|+|\zeta| = p+q$.  Now take the coefficients of \ts
$x_1\ldots x_{p}\ts y_1\ldots y_{p}$ \ts on both sides of~\eqref{eq:lr-schur},
and include only the term of the summation corresponding to $\eta=\be$ and $\zeta=\ga$.
We obtain
$$f^{\al} \, \ge \, \sum_{\eta\vdash p, \ts \zeta\vdash q} \. c^{\al}_{\eta\zeta} \.
f^{\eta}\ts f^{\ze} \, \ge \, c^{\al}_{\beta\gamma} \. f^{\beta}\ts f^{\gamma}\,,
$$
as desired.
\end{proof}

\vskip.8cm

\noindent
{\bf Acknowledgements.} \ We are grateful to Sasha Barvinok,
Jes{\'u}s De Loera, Stephen DeSalvo, Richard Stanley and Fabrizio Zanello
for many helpful conversations.  Special thanks to Ernesto Vallejo for
reading the early draft of the paper and help with the  references.  
The first author was partially supported by the NSF grants,
the second by the Simons Postdoctoral Fellowship.

}
\vskip.8cm

% \newpage

%
%%%%%%%%%%%%%%%%%%%%%%%%%%%%%%%%%%%%%%%%%%%%%%%%%%%%%%%%%%%%%%%%%%%%%%%%

%%%%%%%%%%%%%%%%%%%%%%%%%%%%%%%%%%%%%%%%%%%%%%%%%%%%%%%%%%%%%%%%%%%%%%%%

{\footnotesize

\begin{thebibliography}{131331}\label{refpage}

\bibitem[Alm]{A1}
G.~Almkvist,
Partitions into odd, unequal parts,
\emph{J.~Pure Appl. Algebra }~\textbf{38} (1985), 121--126.

\bibitem[AA]{AA}
G.~Almkvist and G.~E.~Andrews,
A Hardy-Ramanujan formula for restricted partitions,
\emph{J.~Number Theory}~\textbf{38} (1991), 135--144.

\bibitem[AV]{AV}
D.~Avella-Alaminos and E.~Vallejo, 
Kronecker products and the RSK correspondence, 
\emph{Discrete Math.}~\textbf{312} (2012), 1476--1486. 

\bibitem[Bar]{Bar1}
A.~Barvinok,
Asymptotic estimates for the number of contingency tables, integer flows,
and volumes of transportation polytopes, \emph{IMRN}~2009, no.~2, 348--385.

\bibitem[Ben]{Ben}
D.~Benson-Putnins,
An Asymptotic Formula for the Number of Integer Points in Multi-Index
Transportation Polytopes, \ts {\tt arXiv:1402.4715}.

\bibitem[BB]{BB} C.~Bessenrodt and C.~Behns,
On the Durfee size of Kronecker products of characters of
the symmetric group and its double covers,
{\em J.~Algebra}~{\bf 280} (2004), 132--144.

\bibitem[BR]{BR}
A.~Berele and A.~Regev,
Hook Young diagrams with applications to combinatorics and to representations
of Lie superalgebras, \emph{Adv. Math.}~\textbf{64} (1987), 118--175.

\bibitem[BBR]{BBR}
F.~Bergeron, R.~Biagioli and M.~Rosas,
Inequalities between Littlewood-Richardson coefficients,
\emph{J.~Combin.\/ Theory, Ser.~A}~\textbf{113} (2006),
567--590.

\bibitem[BOR1]{BOR}
E.~Briand, R.~Orellana and M.~Rosas,
Reduced Kronecker coefficients and counter-examples to
Mulmuley's strong saturation conjecture~SH,
\emph{Comput. Complexity}~\textbf{18} (2009), 577--600.

\bibitem[BOR2]{BOR2}
E.~Briand, R.~Orellana and M.~Rosas,
The stability of the Kronecker product of Schur functions,
\emph{J.~Algebra}~\textbf{331} (2011), 11--27.

\bibitem[Bri]{Bri}
M.~Brion, Stable properties of plethysm: on two conjectures of Foulkes,
\emph{Manuscripta Math.}~\textbf{80} (1993), 347--371.

\bibitem[B\"ur]{Bur}
P.~B\"urgisser, Review of~\cite{MS}, MR2421083 (2009j:68067).

\bibitem[BI]{BI}
P.~B\"urgisser and C.~Ikenmeyer,
The complexity of computing Kronecker coefficients,
in \emph{Discrete Math. Theor. Comput. Sci. Proc.},
Assoc.~DMTCS, Nancy, 2008, 357--368.

\bibitem[CJZ]{CJZ}
E.~R.~Canfield,  S.~Janson and D.~Zeilberger,
The Mahonian probability distribution on words is asymptotically normal,
\emph{Adv.~Appl.\/ Math.}~\textbf{46} (2011), 109--124;
Corrigendum in~\textbf{49} (2012),~77.

\bibitem[CDW1]{CDW-log}
C.~Chindris, H.~Derksen and J.~Weyman,
Counterexamples to Okounkov's log-concavity conjecture,
\emph{Compos.\/ Math.}~\textbf{143} (2007), 1545--1557.

\bibitem[CDW2]{CDW}
M.~Christandl, B.~Doran and M.~Walter,
Computing Multiplicities of Lie Group Representations,
in \emph{Proc.~53rd FOCS}, 2012, 639--648; \ts {\tt arXiv:}{\tt 1204.4379}.

\bibitem[CHM]{CHM}
M.~Christandl, A.~W.~Harrow and G.~Mitchison,
Nonzero Kronecker coefficients and what they tell us about spectra,
\emph{Comm. Math. Phys.}~\textbf{270} (2007), 575--585.

\bibitem[Dha]{Dha}
V.~Dhand,
A combinatorial proof of strict unimodality for $q$-binomial coefficients,
\emph{Discrete Math.}, to appear; \ts 
{\tt arXiv:1402.1199}.

\bibitem[DG]{DG}
P.~Diaconis and A.~Gangolli,
Rectangular arrays with fixed margins, in
\emph{Discrete probability and algorithms},
Springer, New York, 1995, 15--41.

\bibitem[Dvir]{Dvir}
Y.~Dvir, On the Kronecker product of $S_n$ characters,
\emph{J.~Algebra}~\textbf{154} (1993), 125--140.

\bibitem[ER]{ER}
P.~Erd\H{o}s and L.~B.~Richmond,
On graphical partitions,
\emph{Combinatorica}~\textbf{13} (1993), 57--63.

\bibitem[JK]{JK} G.~D.~James and A.~Kerber,
{\em The Representation Theory of the Symmetric Group},
Cambridge U.~Press, Cambridge, 2009.

\bibitem[LPP]{LPP}
T.~Lam, A.~Postnikov and P.~Pylyavskyy,
Schur positivity and Schur log-concavity,
\emph{Amer.\/ Jour.\/ Math.}~\textbf{129} (2007), 1611--1622.

\bibitem[KT]{KT}
A.~Knutson and T.~Tao,
The honeycomb model of {${\rm GL}_n({\bf C})$} tensor
products.~{I}. {P}roof of the saturation conjecture,
\emph{J.~AMS}~\textbf{12} (1999),
1055--1090.

\bibitem[Mac]{Mac}
I.~G.~Macdonald, \emph{Symmetric functions and Hall polynomials}
(Second ed.), Oxford U.~Press, New York, 1995.

\bibitem[Man]{Man}
L.~Manivel, On rectangular Kronecker coefficients,
\emph{J.~Algebraic Combin.}~\textbf{33} (2011), 153--162.

\bibitem[MY]{MY}
H.~Mizukawa and H.-F.~Yamada,
Rectangular Schur functions and the basic representation of affine Lie algebras,
\emph{Discrete Math.}~\textbf{298} (2005), 285--300.

\bibitem[MS]{MS}
K.~D.~Mulmuley and M.~Sohoni,
Geometric complexity theory~II. Towards
explicit obstructions for embeddings among class varieties,
\emph{SIAM J.~Comput.}~\textbf{38} (2008), 1175--1206.

\bibitem[Nar1]{Nar1}
H.~Narayanan,
On the complexity of computing Kostka numbers and Littlewood-Richardson coefficients,
\emph{J.~Algebraic Combin.}~\textbf{24} (2006), 347--354.

\bibitem[Nar2]{Nar2}
H.~Narayanan,
Estimating certain non-zero Littlewood-Richardson coefficients, 
in \emph{Proc.\/ FPSAC 2014}; \ts 
{\tt arXiv:}{\tt 1306.4060}.

\bibitem[O'H]{O}
K.~M.~O'Hara, Unimodality of Gaussian coefficients: a constructive proof,
\emph{J.~Combin.~Theory, Ser.~A}~\textbf{53} (1990), 29--52.

\bibitem[OR]{OR}
A.~M.~Odlyzko and L.~B.~Richmond,
On the unimodality of some partition polynomials,
\emph{European J.~Combin.}~\textbf{3} (1982), 69--84.

\bibitem[Oko]{Oko}
A.~Okounkov,
Why would multiplicities be log-concave?, in
\emph{Progr.\/ Math.}~\textbf{213},
Birkh\"{a}user, Boston, MA, 2003, 329--347.

\bibitem[PP1]{PP_s}
I.~Pak and G.~Panova, Strict unimodality of $q$-binomial coefficients,
\emph{C.\ts\/R.\/ Math.\/ Acad.\/ Sci.\/ Paris} \textbf{351} (2013), 415--418.

\bibitem[PP2]{PP}
I.~Pak and G.~Panova,
Unimodality via Kronecker products, \emph{J.~Algebraic Combin.}, to appear; 
\ts {\tt arXiv:1304.5044}.

\bibitem[PP3]{PP_c}
I.~Pak and  G.~Panova, On the complexity of computing Kronecker coefficients, \ts
{\tt arXiv:}{\tt 1404.0653}.

\bibitem[PP4]{PP-future}
I.~Pak and G.~Panova, Combinatorics and complexity of Kronecker coefficients,
a survey in preparation.

\bibitem[PPV]{PPV}
I.~Pak, G.~Panova and E.~Vallejo,
Kronecker products, characters, partitions, and the tensor square conjectures,
\ts {\tt arXiv:}{\tt 1304.0738}.

\bibitem[PV]{PV2}
I.~Pak and E.~Vallejo,
Reductions of Young tableau bijections,
\emph{SIAM J.~Discrete Math.}~\textbf{24} (2010), 113--145.

\bibitem[Rem]{Rem}
J.~B. Remmel,
A formula for the Kronecker products of Schur functions of hook shapes,
J.~Algebra~\textbf{120} (1989), 100--118.

\bibitem[RW]{RW}
J.~B. Remmel and T.~Whitehead,
On the Kronecker product of Schur functions of two row shapes,
\emph{Bull.~Belg. Math.~Soc.~Simon Stevin}~\textbf{1} (1994), 649--683.

\bibitem[Ros]{Ros} M.~H.~Rosas,
The Kronecker product of Schur functions indexed by two row shapes of hook shapes,
\emph{J.~Algebraic Combin.}~\textbf{14} (2001), 153--173.

\bibitem[RS]{RS}
K.~F.~Roth and S.~Szekeres,
Some asymptotic formulae in the theory of partitions,
\emph{Quart.\/ J.~Math.}~\textbf{5} (1954), 241--259.

\bibitem[Sna]{Sna}
E.~Snapper, Group characters and nonnegative integral matrices, 
\emph{J.~Algebra}~\textbf{19} (1971), 520--535.

\bibitem[Sta1]{Sta-unim}
R.~P.~Stanley,
Log-concave and unimodal sequences in algebra, combinatorics, and geometry,
in \emph{Ann. New York Acad. Sci.}~\textbf{576}, New York Acad. Sci., New York, 1989,
500--535.

\bibitem[Sta2]{Sta}
R.~P.~Stanley, \emph{Enumerative Combinatorics}, Vol.~2,
Cambridge U.~Press, Cambridge, 1999.

\bibitem[SZ]{SZ}
R.~P.~Stanley and F.~Zanello,
Unimodality of partitions with distinct parts inside Ferrers shapes,
{\tt arXiv:1305.6083}.

\bibitem[Syl]{Syl}
J.~J.~Sylvester,
Proof of the hitherto undemonstrated Fundamental Theorem of Invariants,
\emph{Phil.\/ Mag.}~\textbf{5} (1878), 178--188; reprinted
in \emph{Coll.~Math.~Papers}, vol.~3, Chelsea, New York, 1973,
117--126. %; available at \ts {\tt http://tinyurl.com/c94pphj}

\bibitem[Tak]{Tak}
L.~Tak\'{a}cs,
Some asymptotic formulas for lattice paths,
\emph{J.~Stat.\/ Plann. Inf.}~\textbf{14} (1986),
123--142.

\bibitem[Val1]{Val1}
E.~Vallejo,
Stability of Kronecker products of irreducible characters of the symmetric group,
\emph{Electron.~J.~Combin.}~\textbf{6} (1999), RP~39, 7~pp.

\bibitem[Val2]{Val2}
E.~Vallejo, A diagrammatic approach to Kronecker squares, {\tt arXiv:1310.8362}.

\bibitem[Ves]{Ves}
R.~Vessenes,
\emph{Generalized Foulkes' conjecture and tableaux construction},
Ph.D.~Thesis, California Institute of Technology, 2004, 199 pp.

\bibitem[West]{West}
D.~B.~West, A symmetric chain decomposition of $L(4,n)$,
\emph{European J.~Combin.}~\textbf{1} (1980), 379--383.

\bibitem[Zan]{Zan}
F.~Zanello,
Zeilberger's KOH theorem and the strict unimodality of $q$-binomial coefficients,
\emph{Proc.~AMS}, to appear; {\tt arXiv:1311.4480}.

\bibitem[Zel]{Zel}
A.~Zelevinsky,
Littlewood-Richardson semigroups, in
\emph{New Perspectives in Algebraic Combinatorics},
Cambridge U.~Press, Cambridge, 1999, 337--345.

\end{thebibliography}

}
\end{document}